\documentclass{amsart}

\usepackage{thmtools}
\usepackage{thm-restate}
\usepackage{float}
\usepackage[usenames,dvipsnames]{color}
\usepackage{mathtools}
\newtheorem{theorem}{Theorem}[section]
\usepackage{amssymb}
\usepackage{hyperref}
\usepackage{graphicx} 
\usepackage[rightcaption]{sidecap}
\usepackage{wrapfig}
\usepackage{caption}
\usepackage{subcaption}
\usepackage{tikz-cd}
\usepackage{pinlabel}

\theoremstyle{definition}
\newtheorem{definition}[theorem]{Definition}

\newtheorem{lemma}[theorem]{Lemma}

\newtheorem{corollary}[theorem]{Corollary}
\newtheorem{remark}[theorem]{Remark}
\newtheorem{proposition}[theorem]{Proposition}

\graphicspath{ {./Figures/} }

\begin{document}

\title{$n$-knots in $S^n\times S^2$ and contractible $(n+3)$-manifolds}

\author{Geunyoung Kim}
\address{Department of Mathematics, University of Georgia, Athens, GA 30602}
\curraddr{}
\email{g.kim@uga.edu}
\thanks{}

\begin{abstract}
In $1961$, Mazur \cite{mazur1961note} constructed a contractible, compact, smooth $4$-manifold with boundary which is not homeomorphic to the standard $4$-ball, using a $0$-handle, a $1$-handle and a $2$-handle.  In this paper, for any integer $n\geq2,$ we construct a contractible, compact, smooth $(n+3)$-manifold with boundary which is not homeomorphic to the standard $(n+3)$-ball, using a $0$-handle, an $n$-handle and an $(n+1)$-handle. The key step is the construction of an interesting knotted $n$-sphere in $S^n\times S^2$ generalizing the Mazur pattern.
\end{abstract}

\date{}
\dedicatory{}

\maketitle

\section{Introduction}
In this paper, we prove the following two main theorems.

\begin{restatable}{theorem}{mainone}\label{thm:main1}
For any integer $n\geq2,$ there exists a contractible, compact, smooth $(n+3)$-manifold with boundary admitting a handle decomposition with a $0$-handle, an $n$-handle and an $(n+1)$-handle which is not homeomorphic to the standard $(n+3)$-ball $B^{n+3}$.
\end{restatable}

In order to prove Theorem \ref{thm:main1}, given $n\geq2$, we first construct an interesting $n$-knot $K^n$ in $S^n\times S^2$ (Definition \ref{dfn:theNknt}) which is homotopic but not isotopic to $S^n\times\{y_0\}$, where $y_0\in S^2$ (Proposition \ref{pro:hmtpctoSnXpt} and Corollary \ref{cor:NIstpctoSnXpt}). Secondly we construct a contractible, compact, smooth $(n+3)$-manifold $X_{K^n}$ (Definition \ref{def:MainMfld}) from $S^n\times B^3$ ($0$-handle $\cup$ $n$-handle) by attaching a single $(n+3)$-dimensional $(n+1)$-handle along the $n$-knot $K^n$ in $S^n\times S^2=\partial(S^n\times B^3)$. Finally we prove that $X_{K^n}$ is contractible by showing that $X_{K^n}\times B^1$ is diffeomorphic to $B^{n+4}$ (Proposition \ref{pro:XprdctIntvldiffeotoball}) and prove that $X_{K^n}$ is not homeomorphic to $B^{n+3}$ by showing that $\partial X_{K^n}$ is a non-simply connected homology $(n+2)$-sphere (Corollary \ref{cor:nonsmplycnctdhlgysphre}).

\begin{restatable}{theorem}{maintwo}\label{thm:main2}
For any integer $n\geq2,$ there exists an involution of $S^{n+3}$ whose fixed point set is a non-simply connected homology $(n+2)$-sphere.
\end{restatable}

In order to prove Theorem \ref{thm:main2}, given $n\geq2$, we first show that the double $DX_{K^n}=X_{K^n}\cup_{id}\overline{X_{K^n}}$ of $X_{K^n}$ is diffeomorphic to $S^{n+3},$ where $id:\partial X_{K^n}\rightarrow \partial X_{K^n}$ is an identity map (Lemma \ref{lem:doubleissphere}). We then define an involution $\phi: S^{n+3}\rightarrow S^{n+3}$ switching copies of $X_{K^n}$ and fixing the non-simply connected homology $(n+2)$-sphere $\partial X_{K^n}.$

\begin{remark} Here we discuss the relationship between our results and earlier results.
    \begin{enumerate}
    \item In \cite{mazur1961note} Mazur proved Theorem \ref{thm:main1} and Theorem \ref{thm:main2} when $n=1$. We can consider Mazur's $1$-knot in $S^1\times S^2$ as the result of surgery of three parallel copies of $S^1\times\{y_0\}\subset S^1\times S^2$ along two $2$-dimensional $1$-handles whose cores are trivial and with some twistings (See $J_1$ for Mazur's $1$-knot and $J_2$ for another interesting $1$-knot in Remark \ref{rmk:intrstng1knts}). However, we cannot generalize Mazur's $1$-knot to an $n$-knot in $S^n\times S^2$ obtained from three parallel copies of $S^n\times\{y_0\}\subset S^n\times S^2$ by surgery along two $(n+1)$-dimensional $1$-handles whose cores are trivial when $n\geq2$ because the resulting $n$-knot is always isotopic to $S^n\times \{y_0\}\subset S^n\times S^2.$ In Definition \ref{dfn:theNknt}, we resolve this issue and find interesting $(n+1)$-dimensional $1$-handles whose cores are non-trivial but very simple so that we construct the $n$-knot $K^n$ in $S^n\times S^2$.
    \item In \cite{sato19912} Sato proved Theorem \ref{thm:main1} and Theorem \ref{thm:main2} when $n=2.$ Sato constructed a $2$-knot $F$ in $S^2\times S^2$ which is homotopic but not isotopic to $S^2\times \{y_0\}$ by surgery along a simple closed curve in the complement of the $5$-twist spun trefoil in $S^4$ in \cite{sato1991locally}. However, Sato's construction is not very explicit so it is difficult to visualize the $2$-knot $F$. In particular, we do not know the geometric intersection number $|F\cap \{x_0\}\times S^2|$ directly and it is hard to see why $F\times\{0\}\subset S^2\times S^2\times B^1$ is isotopic to $S^2\times\{y_0\}\times\{0\}\subset S^2\times S^2\times B^1$. Our construction of $K^n$ resolves these issues.
    \item In \cite{kervaire1969smooth}, 
    Kervaire proved the existence of contractible, compact, smooth $(n+3)$-manifolds which are not homeomorphic to $B^{n+3}$, where $n\geq2.$ However, the proof does not tell us much about the handle decomposition or give us a proof of Theorem \ref{thm:main2}.
    \end{enumerate}
\end{remark} 

\begin{remark} Here we note some nice properties of our $n$-knot $K^n$ in $S^n\times S^2$, and highlight some ways in which our construction is an improvement on the techniques used in the results described above.
    \begin{enumerate}
    \item The construction of $K^n$ is very explicit and visualized for every $n\geq2$ (Definition \ref{dfn:theNknt}). We may construct infinitely many interesting $n$-knots in $S^n\times S^2$ by modifying the cores of $(n+1)$-dimensional $1$-handles. For example, the cores used to construct $K^n$ come from the case when $m_1=1, m_2=-1,$ and $m_3=\cdots =m_{2i}=0$ in the left-hand side of Figure \ref{A3}. We may then construct infinitely many contractible, compact, smooth $(n+3)$-manifolds which are not homeomorphic to $B^{n+3}$.
    \item $K^n$ is the simplest example of this construction in the sense that the geometric intersection number $|K^n\cap (\{x_0\}\times S^2)|=3$ and the algebraic intersection number $K^n\cdot (\{x_0\}\times S^2)=1,$ where $x_0\in S^n$ (Proposition \ref{pro:hmtpctoSnXpt}). Furthermore it is impossible to have $|F\cap \{x_0\}\times S^2|<3$ for any $n$-knot $F$ isotopic to $K^n$ (Corollary \ref{cor:mniIntrsctnNmbrs}).
    \item A homotopy between $K^n\subset S^n\times S^2$ and $S^n\times \{y_0\}\subset S^n\times S^2$ is visualized (Proposition \ref{pro:hmtpctoSnXpt}).
    \item An isotopy between $K^n\times\{0\}\subset S^n\times S^2\times B^1$ and $S^n\times\{y_0\}\times\{0\}\subset S^n\times S^2\times B^1$ is visualized (Proposition \ref{pro:istpctoSnXptXpt}). This isotopy is essential to proving that $X_{K^n}\times B^1$ is diffeomorphic to $B^{n+4}$ (Proposition \ref{pro:XprdctIntvldiffeotoball}).
    \item The construction of $K^n$ gives an explicit handle decomposition of $S^n\times S^2\setminus int(\nu(K^n))$ (Remark \ref{rem:HDofCmplmnt}) so we can easily find the fundamental group  $\pi_1(S^n\times S^2\setminus int(\nu(K^n)))$ (Proposition \ref{pro:FndmntlGpofCmplmnt}).
    \item The construction of $K^n$ gives an explicit handle decomposition of the non-simply connected homology $(n+2)$-sphere $\partial X_{K^n}$ (Remark \ref{rem:HDofbndry}) so we can easily show that $\pi_1(X_{K^n})\cong \pi_1(S^n\times S^2\setminus int(\nu(K^n)))$ (Proposition \ref{pro:FndmntlGrSame}).
    \end{enumerate}
\end{remark}

\subsection*{Organization}
In section \ref{s1} we set up some standard notation and interpret Mazur's $1$-knot in $S^1\times S^2$ from the point of view of surgery as motivation for our construction of the $n$-knot $K^n$ in $S^n\times S^2.$ In section \ref{s2} we construct our $n$-knot $K^n$ in $S^n\times S^2$ and the contractible $(n+3)$-manifold $X_{K^n}$, and then prove Theorem \ref{thm:main1} and Theorem \ref{thm:main2}.

\subsection*{Conventions}
In this paper, we work in the smooth category. $\mathbb{R}^n, S^n, B^n$ and $\{0\}\subset B^n\subset \mathbb{R}^n$ stand for the real $n$-space, the standard $n$-sphere, the standard $n$-ball and the origin of $B^n$ or $\mathbb{R}^n$.

\subsection*{Acknowledgements} This work was supported in part by National Science Foundation grant DMS-2005554 ``Smooth $4$--Manifolds: $2$--, $3$--, $5$-- and $6$--Dimensional Perspectives''.

\begin{figure}[ht]
    \centering
    \includegraphics[width=0.5\textwidth]{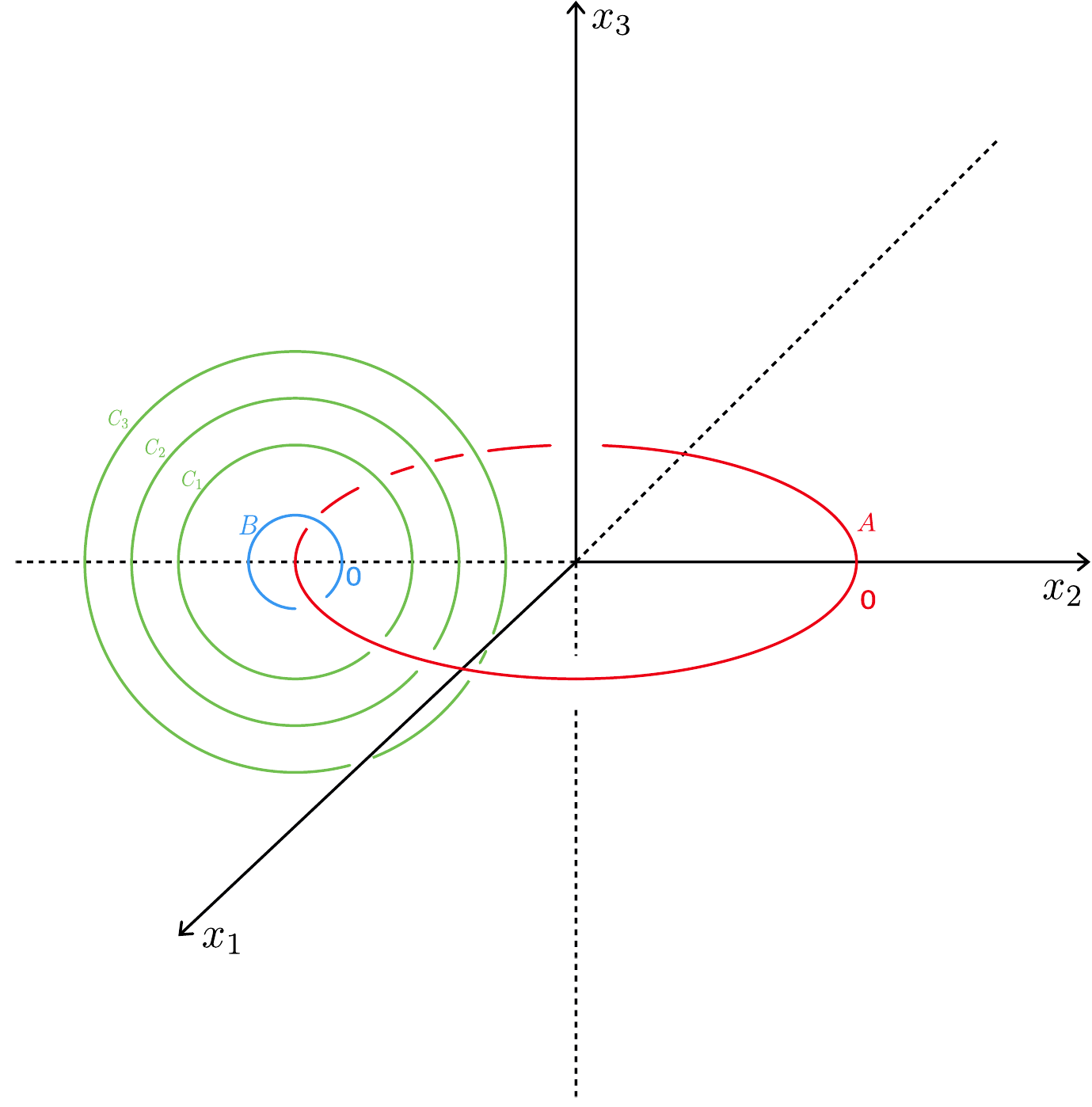}
    \caption{Standard handle decomposition of $S^n\times S^2$ and three parallel copies of $S^n\times\{y_0\}$. $A$ is the attaching sphere of the $2$-handle with $0$-framing, $B$ is the attaching sphere of $n$-handle with $0$-framing, and $C_t$ is the equator of $S^n\times\{y_t\}$.}
    \label{A1}
\end{figure}

\section{Preliminaries}\label{s1}
We begin by explicitly describing the standard handle decomposition of $S^n\times S^2$ and the associated attaching maps.

\begin{remark}[Standard handle decomposition of $S^n\times S^2$]\label{rem:StndrdHdlDcmpstnSnXS2}
Decompose $S^n=B^n_{-}\cup B^n_{+}$ into two $n$-dimensional balls and $S^2=B^2_{-}\cup B^2_{+}$ into two $2$-dimensional balls. Then $S^n\times S^2=(B^n_{-}\cup B^n_{+})\times (B^2_{-}\cup B^2_{+})=(B^n_{-}\times B^2_{-})\cup(B^n_{-}\times B^2_{+})\cup(B^n_{+}\times B^2_{-})\cup(B^n_{+}\times B^2_{+})$ has a handle decomposition with a single $0$-handle, a single $2$-handle, a single $n$-handle, and a single $(n+2)$-handle. We can easily see the attaching sphere of the $2$-handle (trivial $1$-knot) and the attaching sphere of the $n$-handle (trivial $(n-1)$-knot) on the boundary of the $0$-handle i.e., $(\{0\}\times S^1)\cup (S^{n-1}\times \{0\})\subset \partial (B^{n}_{-}\times B^2_{-})\cong \partial B^{n+2}=S^{n+1}.$ For future reference, we parameterize the trivial $1$-knot and the trivial $(n-1)$-knot in $\mathbb{R}^{n+1}\subset \mathbb{R}^{n+1}\cup \{\infty\}\cong S^{n+1}$. The trivial $1$-knot $\{0\}\times S^1$ corresponds to $A:=\{(x_1,\dots,x_{n+1})\in \mathbb{R}^{n+1} \mid {x_1}^2+{x_2}^2=1, x_i=0$ for $i>2\}$ and the trivial $(n-1)$-knot $S^{n-1}\times\{0\}$ corresponds to $B:=\{(x_1,\dots,x_{n+1})\in \mathbb{R}^{n+1} \mid x_1=0,{(x_2+1)}^2+\sum_{i=3}^{n+1} {x_i}^2 = {(\frac{1}{5}})^2\}.$ An $(\mathbb{R}^3\times\{0\})$-slice $(A\cup B)\cap (\mathbb{R}^3\times\{0\})$ of $A\cup B$ is the Hopf link in Figure \ref{A1}, where $\mathbb{R}^3\times\{0\}\subset \mathbb{R}^3\times \mathbb{R}^{n-2}\cong \mathbb{R}^{n+1}.$ Again, $S^n\times S^2$ can be recovered from $B^{n+2}$ by attaching a single $2$-handle to $S^{n+1}\cong \mathbb{R}^{n+1}\cup \{\infty\}$ along $A$ with $0$-framing (product framing), a single $n$-handle with $0$-framing (product framing), and a single $(n+2)$-handle (we don't draw the $(n+2)$-handle). The $0$ in Figure \ref{A1} is shorthand for the obvious product framing. For example, Figure \ref{A1} is a Kirby diagram of $S^2\times S^2$ when $n=2.$ A parallel copy $C_t:=\{(x_1,\dots,x_{n+1})\in \mathbb{R}^{n+1} \mid x_1=0,{(x_2+1)}^2+\sum_{i=3}^{n+1} {x_i}^2 = {(\frac{t}{4}})^2\}$ of $B$ bounds a properly embedded trivial $n$-ball $D^{-}_{t}$ in $B^{n+2}$ ($C_t$ bounds a trivial $n$-ball in $S^{n+1}$ and push the interior of the $n$-ball into the interior of $B^{n+2}$) and a copy of the core of the $n$-handle $D^{+}_{t}$ so $C_{t}=D^{-}_{t}\cap D^{+}_{t}$ represents the equator of $S^n\times \{y_t\}$ and $D_{t}:=D^{-}_{t}\cup D^{+}_{t}$ represents $S^n\times \{y_t\}.$  From now on, we use red, blue and green for $A,B$, and $C_t,$ respectively.
\end{remark}

Next we establish terminology for handles embedded in an ambient manifold and attached to a submanifold.

\begin{definition}
Let $n\geq1$. Let $N^n\subset M^{n+2}$ be a $n$-dimensional submanifold of $(n+2)$-dimensional manifold $M$. An $(n+1)$-dimensional submanifold $h\subset M$ is called a \textit{$1$-handle attached to $N$} if there exists an embedding $e:B^1\times B^n\hookrightarrow M$ such that $h=e(B^1\times B^n)$ and $h\cap N=e(\partial B^1\times B^n).$ We call $C_h:=e(B^1\times\{0\})$ the \textit{core} of $h$ and $N_h:=(N\setminus int(e(\partial B^1 \times B^{n})))\cup e(B^1\times \partial B^{n})\subset M$ \textit{the result of surgery} of $N$ along $h$ in $M.$ Here, $N_h$ is assumed to be oriented so that the orientation of $N\setminus int(e(\partial B^1\times B^n))$ extends to the orientation of $N_h$.
\end{definition}

Boyle proved the following when $n=2$ and $M=S^4.$ See \cite{boyle1988classifying} for more details. 

\begin{proposition}\label{pro:Istpy1HndlsWSCrs}
Let $n\geq2.$ Let $h$ and $h'$ be $1$-handles attached to $N^n\subset M^{n+2}$ with the same core $C.$ Then there exists an ambient isotopy of $M$ taking $h$ to $h'$ fixing $N$ setwise. Furthermore, the results of surgery $N_h$ and $N_{h'}$ are isotopic.
    \begin{proof}
    Following \cite{boyle1988classifying}, the difference between two such $1$-handles with the same core gives a map $\theta:B^1\rightarrow G_{n+1,n};$ the Grassmannian manifold of oriented $n$-planes in $\mathbb{R}^{n+1},$ with $\theta(-1)=\theta(1).$ Since $n\geq2$ and $G_{n+1,n}\cong S^n$, we have $\pi_1(G_{n+1,n})=0,$ so there exists an isotopy of $M$ taking $h$ to $h'$  fixing $N$ setwise. From this we see that the results of surgery $N_h$ and $N_{h'}$ are isotopic.
    \end{proof}
\end{proposition}

\begin{corollary}\label{cor:Istpy1HndlsWICrs}
Let $n\geq2.$ Let $h$ and $h'$ be $1$-handles attached to $N^n\subset M^{n+2}$ with cores $C_h$ and $C_{h'},$ respectively. If $C_h$ and $C_{h'}$ are isotopic through arcs such that the boundary of each arc is in $N$ and the interior of each arc doesn't intersect with $N,$ then there exists an ambient isotopy of $M$ taking $h$ to $h'$ fixing $N$ setwise. In particular, the results of surgery $N_h$ and $N_{h'}$ are isotopic.
    \begin{proof}
    By the tubular neighborhood theorem and Proposition \ref{pro:Istpy1HndlsWSCrs}, the isotopy taking $C_h$ to $C_{h'}$ through arcs such that the boundary of each arc is in $N$ and the interior of each arc doesn't intersect with $N$ can extend to an ambient isotopy of $M$ taking $h$ to $h'$ fixing $N$ setwise. From this we see that the results of surgery $N_h$ and $N_{h'}$ are isotopic.
    \end{proof}
\end{corollary}

\begin{remark}\hfill
    \begin{enumerate}
    \item Corollary \ref{cor:Istpy1HndlsWICrs} is not true when $n=1$ because $\pi_1(G_{n+1,n})\cong\mathbb{Z}.$ Therefore different framings of a core may give non-isotopic $1$-handles; see Remark \ref{rmk:intrstng1knts}.(4).
    \item When $n\geq2$, a homotopy between arcs implies an isotopy between arcs and there is a unique framing of a core, so $(n+1)$-dimensional $1$-handles are less complicated than $2$-dimensional $1$-handles.
    \end{enumerate}
\end{remark}

\begin{figure}[ht]
    \centering
    \includegraphics[width=0.7\textwidth]{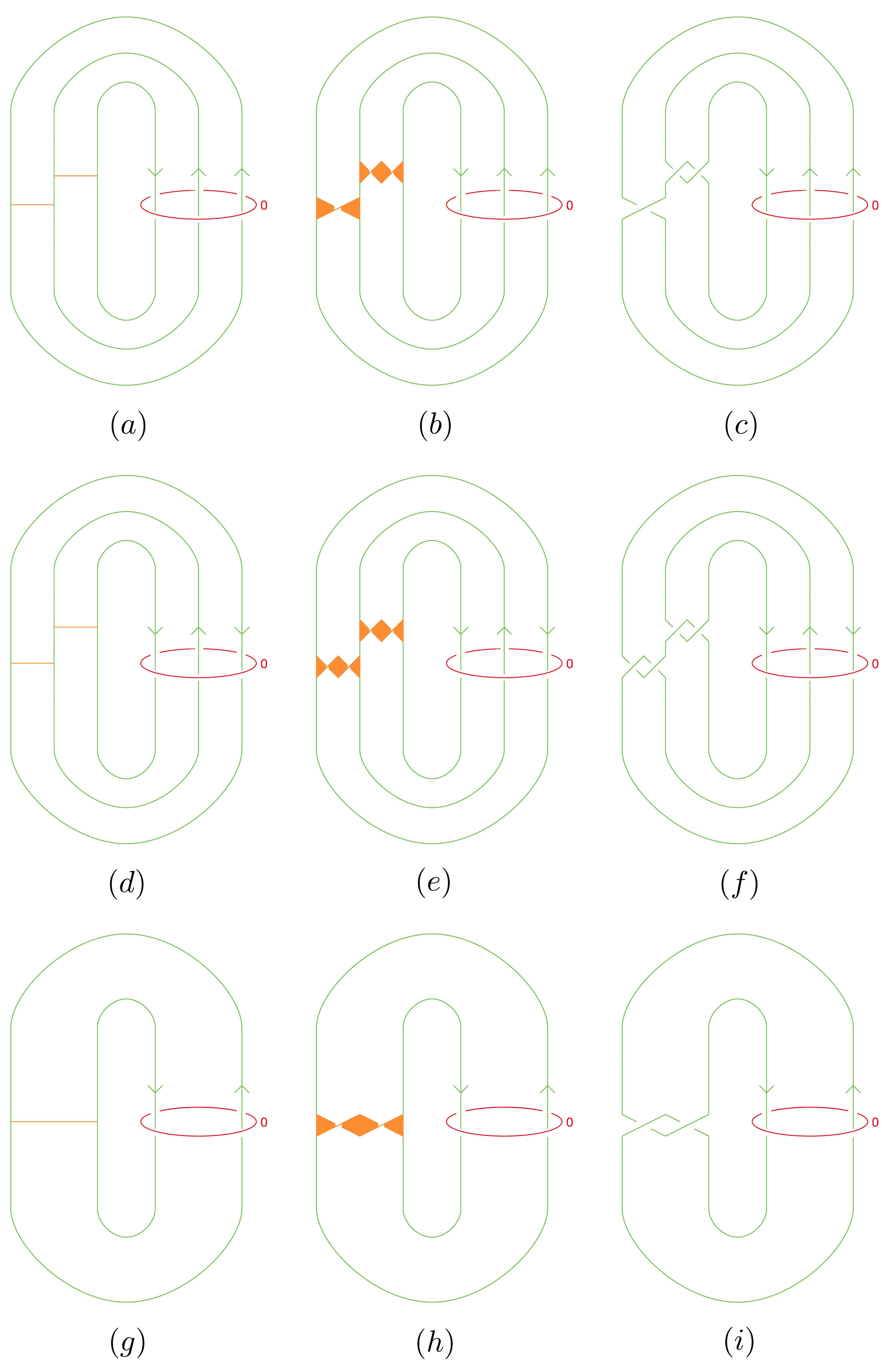}
    \caption{First column $(a),(d),(g)$: parallel copies of $S^1\times \{y_0\}$ in $S^1\times S^2$ with trivial cores of $2$-dimensional $1$-handles. Second column $(b),(e),(h)$: $2$-dimensional $1$-handles determined by the cores and framings (or twistings). Third column $(c),(f),(i)$: results of surgery $J_1,J_2,$ and $J_3$. First row: a process of obtaining $J_1$ by surgery. Second row: a process of obtaining $J_2$ by surgery. Third row: a process of obtaining $J_3$ by surgery.}
    \label{A2}
\end{figure}

We now analyze Mazur's knot $J_1$ and some other interesting knots $J_2$ and $J_3$ in $S^1\times S^2$ from the point of view of surgery along $1$-handles. Figure \ref{A2} illustrates these examples. In this figure we consider $S^1\times S^2$ as the boundary of $S^2\times B^2,$ where $S^2\times B^2$ is obtained from $B^4$ by attaching a $2$-handle along the unknot with the $0$-framing (product framing). Observe the following features of the knots constructed in Figure \ref{A2}:

\begin{remark}\label{rmk:intrstng1knts}\hfill
    \begin{enumerate}
    \item $J_1$ is homotopic but not isotopic to $S^1\times \{y_0\}$ in $S^1\times S^2$; see Figure \ref{A2}.($c$).
    \item $J_2$ is homotopic but not isotopic to $S^1\times \{y_0\}$ in $S^1\times S^2$; see Figure \ref{A2}.($f$).
    \item $J_3$ is homotopic but not isotopic to the unknot in $S^1\times S^2$; see Figure \ref{A2}.($i$).
    \item $J_1, J_2$ and $J_3$ are obtained from parallel copies of $S^1\times \{y_0\}$ by surgery along $2$-dimensional $1$-handles in the figure. Here, we can see that handles are attached so that the results of surgery are oriented and depend on the cores and framings (or twistings) of the cores. (See the second column in Figure \ref{A2}, and note that we may obtain different knots by twisting the bands more.)
    \item In the next section we will construct an $n$-knot in $S^n\times S^2$ from three parallel copies of $S^n\times\{y_0\}$ by surgery along two interesting $(n+1)$-dimensional $1$-handles.
    \end{enumerate}
\end{remark}

\begin{figure}[ht]
    \centering
    \includegraphics[width=0.85\textwidth]{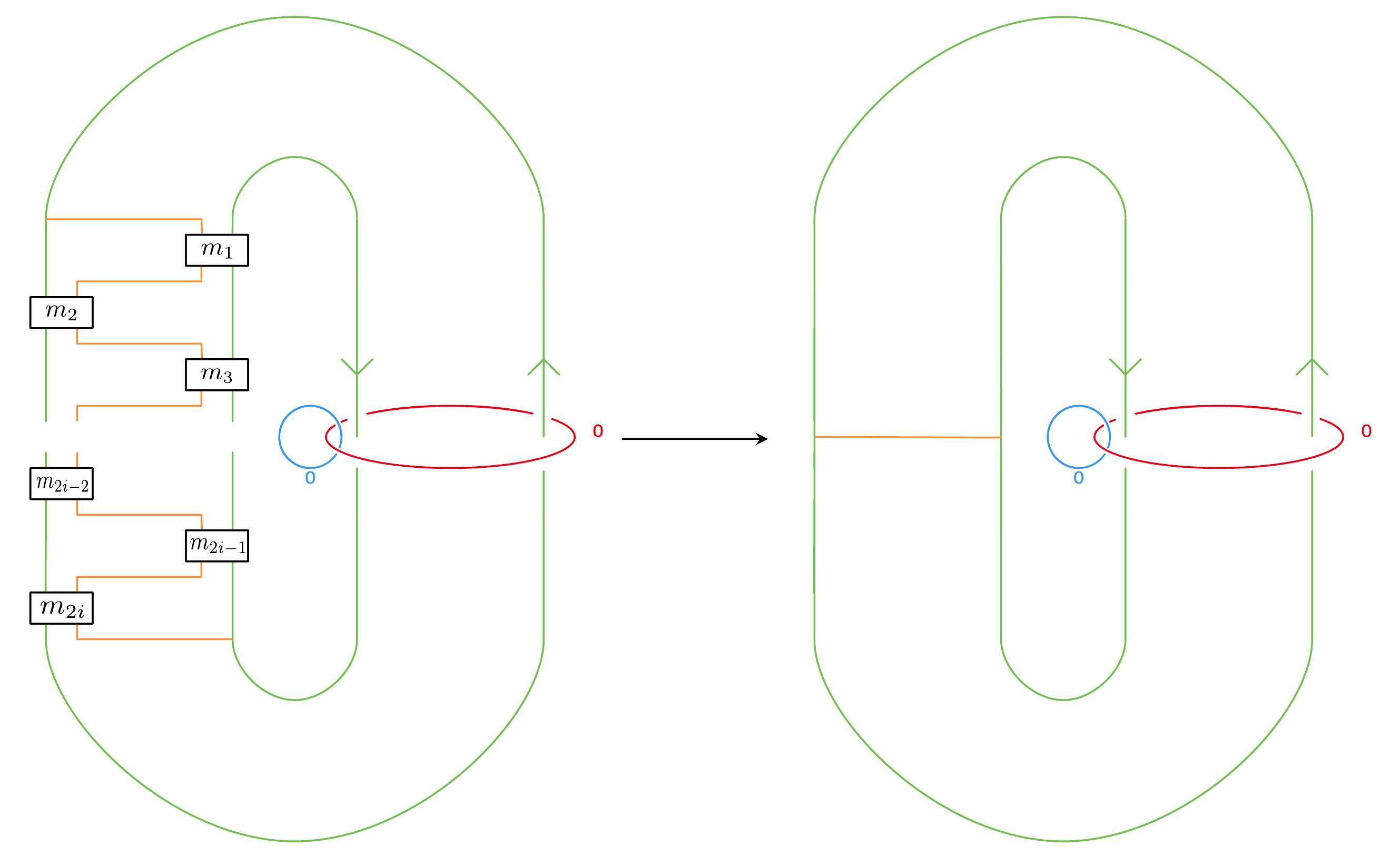}
    \caption{An isotopy between cores. An integer $m$ in the box indicates $m$-full positive twist. $C_i$ is the equator of $D_i$. Left: any arc attached to $C_1\cup C_2$ is isotopic to the orange arc for some values of $m_1,\dots, m_{2i}$. Right: the trivial arc attached to $C_1\cup C_2$. }
    \label{A3}
\end{figure}

A natural question related to the construction of $J_3$ is whether one can construct an $n$-knot in $S^n\times S^2$ which is homotopic but not isotopic to the unknot from two parallel copies of $S^n\times \{y_0\}$ by surgery along a single $(n+1)$-dimensional $1$-handle. However, the following theorem shows that this does not work for $n\geq2.$

\begin{figure}[ht]
    \centering
    \includegraphics[width=0.5\textwidth]{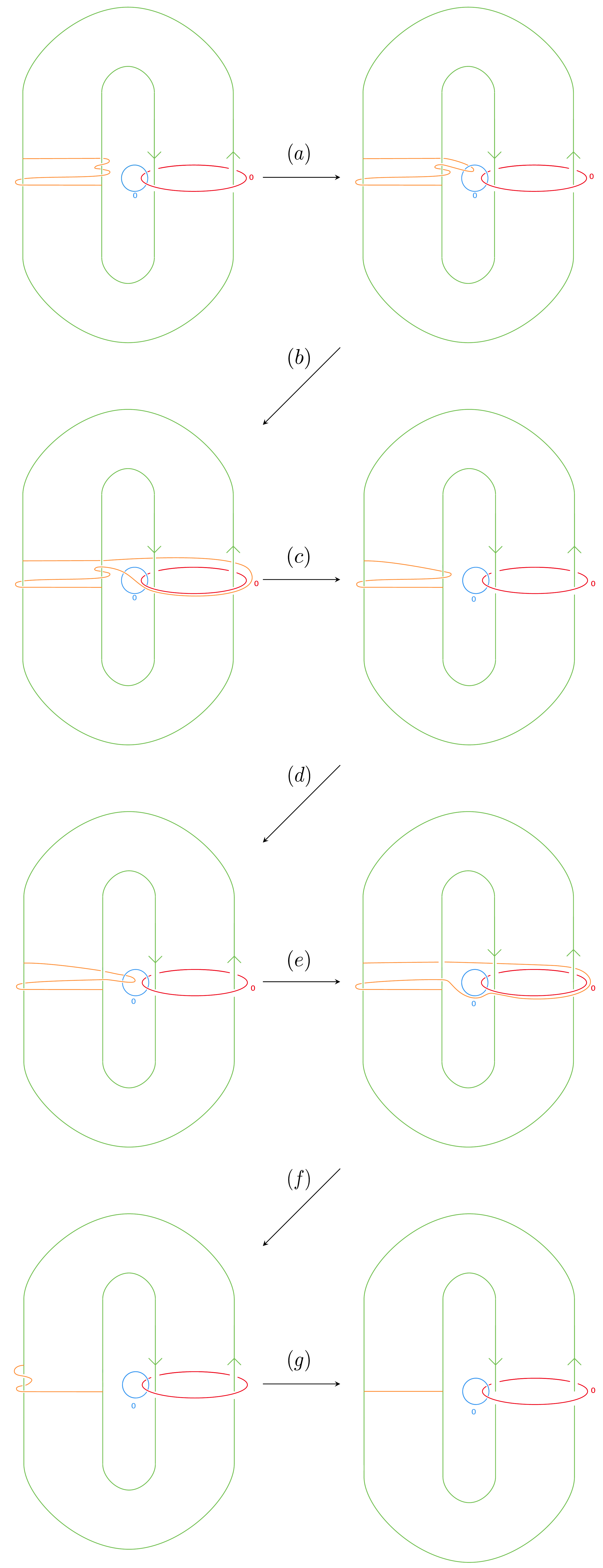}
    \caption{An isotopy between cores when $m_1=2$ and $m_2=-1.$ $(a),(d)$: isotopies pushing the orange arc into $B^{n+2}$ and pulling back. $(b),(e)$: isotopies sliding the orange arc over $2$-handle. $(c),(f),(g)$: obvious isotopies.}
    \label{A4}
\end{figure}

\begin{theorem}
Fix $n\geq2$. Let $N=(S^n\times\{y_1\}) \cup (\overline{S^n\times\{y_2\}})\subset S^n\times S^2$ with opposite orientations, where $y_1\neq y_2\in S^2$. Let $h:=e(B^1\times B^n)$ be a $1$-handle attached to $N$ for some embedding $e:B^1\times B^{n}\hookrightarrow S^n\times S^2$ such that $(S^n\times \{y_1\}) \cap h=e(\{-1\}\times B^{n})$ and $(\overline{S^n\times \{y_2\}}) \cap h=e(\{1\}\times B^{n})$. Then the result of surgery $N_h$ is isotopic to the unknot, i.e., $N_h$ bounds an $(n+1)$-ball in $S^n\times S^2$.
    \begin{proof}
    Consider the standard handle decomposition of $S^n\times S^2$ described in Remark \ref{rem:StndrdHdlDcmpstnSnXS2} and Figure \ref{A1}. Let $D_1=S^n\times\{y_1\}$ and $D_2=\overline{S^n\times\{y_2\}}.$ Now consider a $1$-handle $h$ attached to $D_1\cup D_2.$ By Corollary \ref{cor:Istpy1HndlsWICrs}, it suffices to consider the core of the $1$-handle $h$. The core of the $1$-handle can be isotoped into $\mathbb{R}^3\times\{0\}\subset \mathbb{R}^3\times \mathbb{R}^{n-2}\subset \mathbb{R}^{n+1}\cup \{\infty\}$ and furthermore isotoped into the orange arc for some integers $m_1,\dots, m_{2i}$ in the left-hand side of Figure \ref{A3} because a homotopy between arcs implies an isotopy of arcs in $S^n\times S^2$. We will show that the orange arc in the left-hand side of Figure \ref{A3} can be isotoped into the trivial arc in the right-hand side of Figure \ref{A3}. Figure \ref{A4} illustrates how to isotope the orange arc into the trivial when $m_1=2$ and $m_2=-1$.  Repeating the process illustrated in this example shows how to do the general case. Therefore the result of surgery $N_h$ along $h$ is isotopic to the result of surgery along the trivial $1$-handle which is the unknot in $S^n\times S^2$ by Corollary \ref{cor:Istpy1HndlsWICrs}.
    \end{proof}
\end{theorem}

\begin{figure}[ht]
    \centering
    \includegraphics[width=0.45\textwidth]{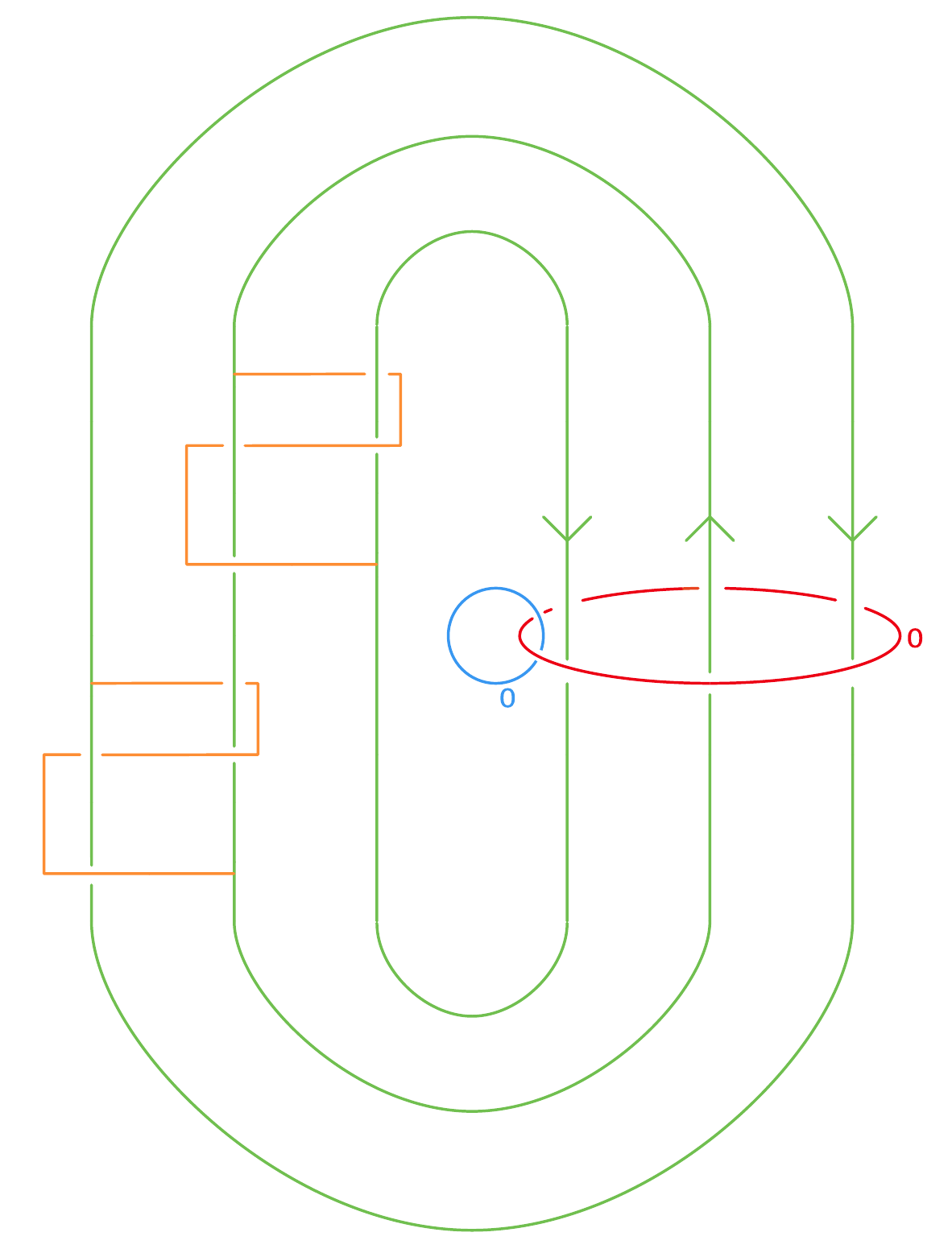}
    \caption{$K^n$ is the result of surgery of $(S^n\times\{y_1\}) \cup (\overline{S^n\times\{y_2\}}) \cup (S^n\times\{y_3\})$ along two $1$-handles $h_{12}$ and $h_{23}$ with orange cores.}
    \label{A5}
\end{figure}

\section{Main theorems}\label{s2}

\begin{definition}\label{dfn:theNknt}
Let $n\geq2$. Let $N=(S^n\times\{y_1\}) \cup (\overline{S^n\times\{y_2\}}) \cup (S^n\times\{y_3\})\subset S^n\times S^2$, where $y_1,y_2,y_3\in S^2$ are three distinct points. Let $h_{12}$ be the $1$-handle attached to $(S^n\times\{y_1\}) \cup (\overline{S^n\times\{y_2\}})$ whose core is in Figure \ref{A5}. Let  $h_{23}$ be the $1$-handle attached to $(\overline{S^n\times\{y_2\}}) \cup (S^n\times\{y_3\})$  whose core is in Figure \ref{A5}. We define an $n$-knot $K^n$ in $S^n\times S^2$ to be the result of surgery of $N$ along $h_{12}\cup h_{23}.$
\end{definition}

\begin{figure}[ht]
    \centering
    \includegraphics[width=0.45\textwidth]{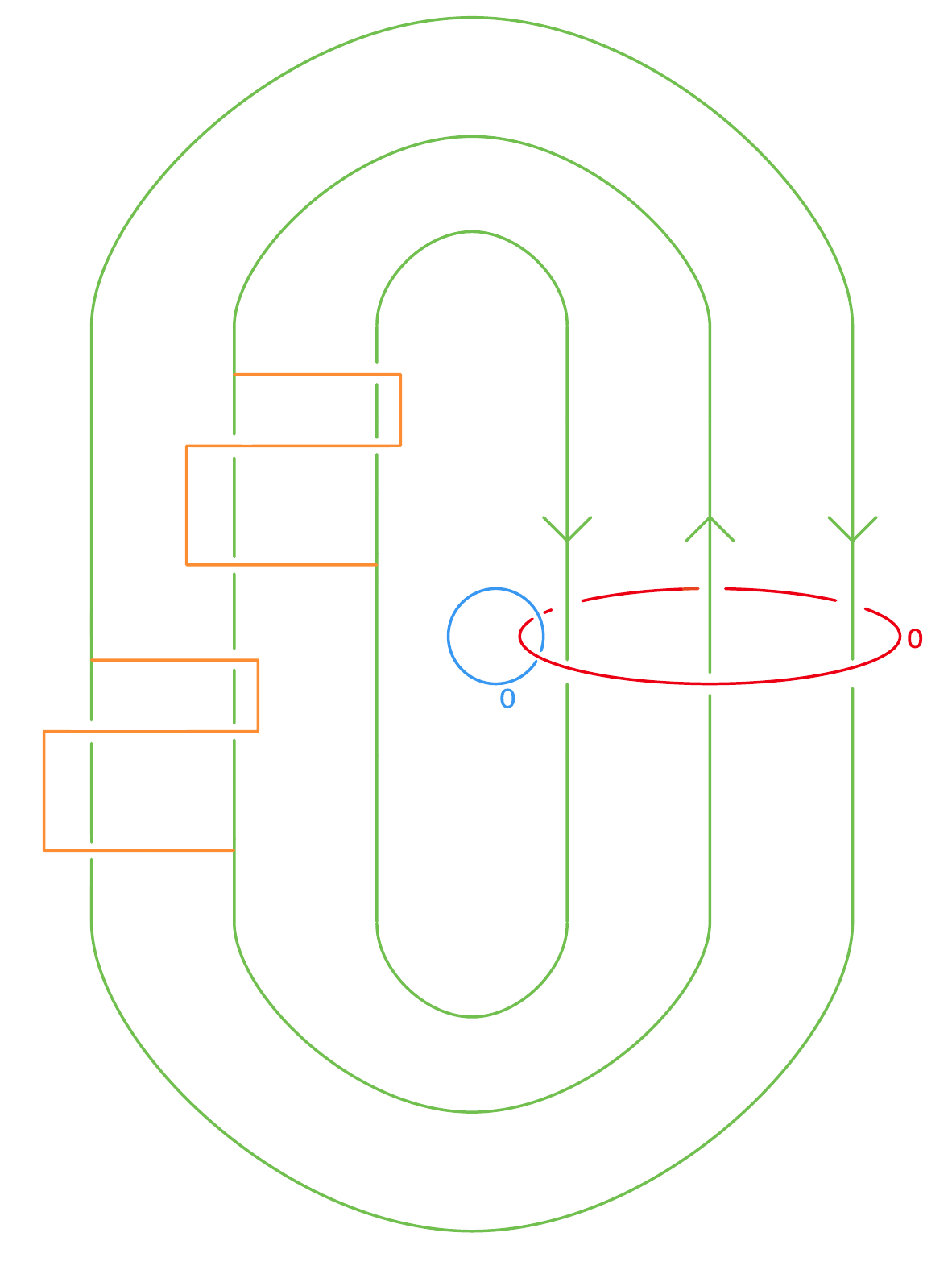}
    \caption {Isotoped cores (trivial cores) in $S^n\times S^2$ or $S^n\times S^2\times B^1$. The isotopy looks like crossing changes in the $(\mathbb{R}^3\times\{0\})$-slice.}
    \label{A6}
\end{figure}
We now see some properties of $K^n.$

\begin{proposition}\label{pro:hmtpctoSnXpt}
$K^n$ is homotopic to $S^n\times \{y_0\}$ in $S^n\times S^2,$ the geometric intersection number $|K^n\cap (\{x_0\}\times S^2)|=3,$ and the algebraic intersection number $K^n\cdot (\{x_0\}\times S^2)=1,$ where $x_0\in S^n.$
    \begin{proof}
    There exists an isotopy between the union of the two cores in Figure \ref{A5} and the union of the two trivial cores in Figure \ref{A6} such that at one moment of the isotopy the arcs intersect the green spheres at four points (like crossing changes). The result of the surgery along the two $1$-handles with the trivial cores in Figure \ref{A6} is isotopic to $S^n\times \{y_0\}$, so $K^n$ is homotopic to $S^n\times \{y_0\}$. Clearly, $|K^n\cap (\{x_0\}\times S^2)|=3,$ and the algebraic intersection number $K^n\cdot (\{x_0\}\times S^2)=1$ from the construction. 
    \end{proof}
\end{proposition}

\begin{proposition}\label{pro:istpctoSnXptXpt}
$K^n\times\{0\}$ in $S^n\times S^2\times B^1$ is isotopic to $S^n\times \{y_0\}\times \{0\}$.
    \begin{proof}
    We can isotope each core of the $1$-handles in Figure \ref{A5} to the trivial core in Figure \ref{A6} using the extra $B^1$ factor without intersections between arcs and green spheres. By Corollary \ref{cor:Istpy1HndlsWICrs}, $K^n\times \{0\}$ is isotopic to $S^n\times \{y_0\}\times \{0\},$ which is isotopic to the result of the surgery along the trivial cores in Figure \ref{A6}.
    \end{proof}
\end{proposition}

\begin{figure}[ht]
    \centering
    \includegraphics[width=0.45\textwidth]{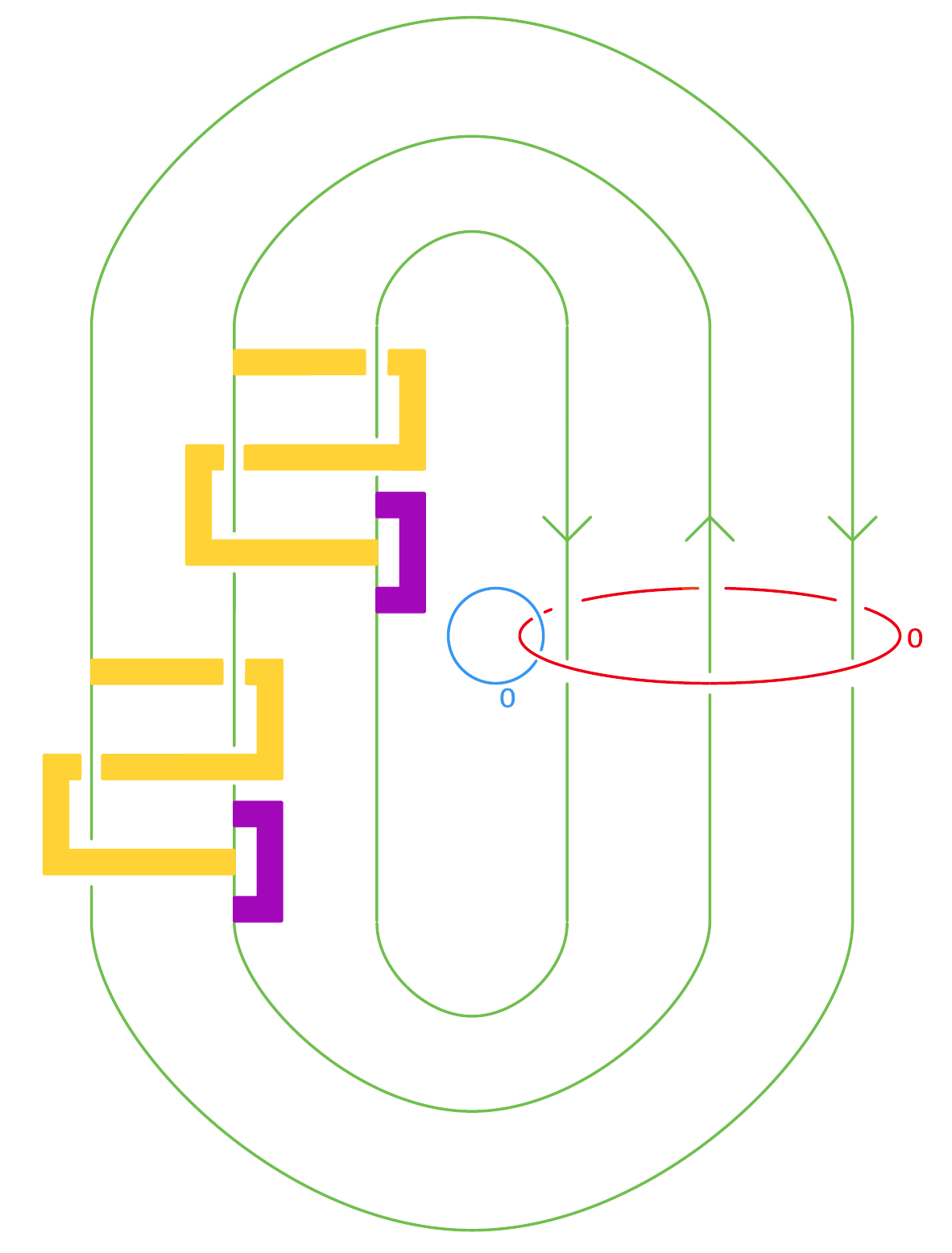}
    \caption{A handle decomposition of $K^n$ in the handle decomposition of $S^n\times S^2$ consists of three $n$-dimensional $0$-handle, two $1$-handles (yellow), two $(n-1)$-handles (purple), and three $n$-handles (not drawn).}
    \label{A7}
\end{figure}

\begin{remark}[A handle decomposition of $K^n$]\label{rem:HDofKn}
$K^n$ is obtained from three parallel copies of $S^n\times \{y_0\}$ by surgery along two $(n+1)$-dimensional $1$-handles in Figure \ref {A5}. Here, the green equator of each parallel copy bounds a properly embedded trivial $n$-ball ($n$-dimensional $0$-handle) in $B^{n+2}$ and a copy of the core ($n$-dimensional $n$-handle) of the $(n+2)$-dimensional $n$-handle of $S^n\times S^2.$ Surgery along a $1$-handle $B^1\times B^n$ removes $(\{-1\} \times B^n) \cup (\{1\}\times B^n)$ and glues $B^1\times S^{n-1}$ so $B^1\times S^{n-1}=B^1\times (B^n_{-}\cup B^n_{+})=(B^1\times B^{n-1}_{-})\cup (B^1\times B^{n-1}_+)$ consists of an $n$-dimensional $1$-handle (yellow) and an $(n-1)$-handle (purple) in Figure \ref {A7}. Therefore, $K^n$ has a handle decomposition with three $0$-handles, two $1$-handles, two $(n-1)$-handles and three $n$-handles. For example, Figure \ref{A7} is a banded unlink diagram for $K^2$ in $S^2\times S^2$ when $n=2.$ See \cite{hughes2020isotopies} for more details of banded unlink diagrams.
\end{remark}

\begin{figure}[ht]
    \centering
    \includegraphics[width=0.45\textwidth]{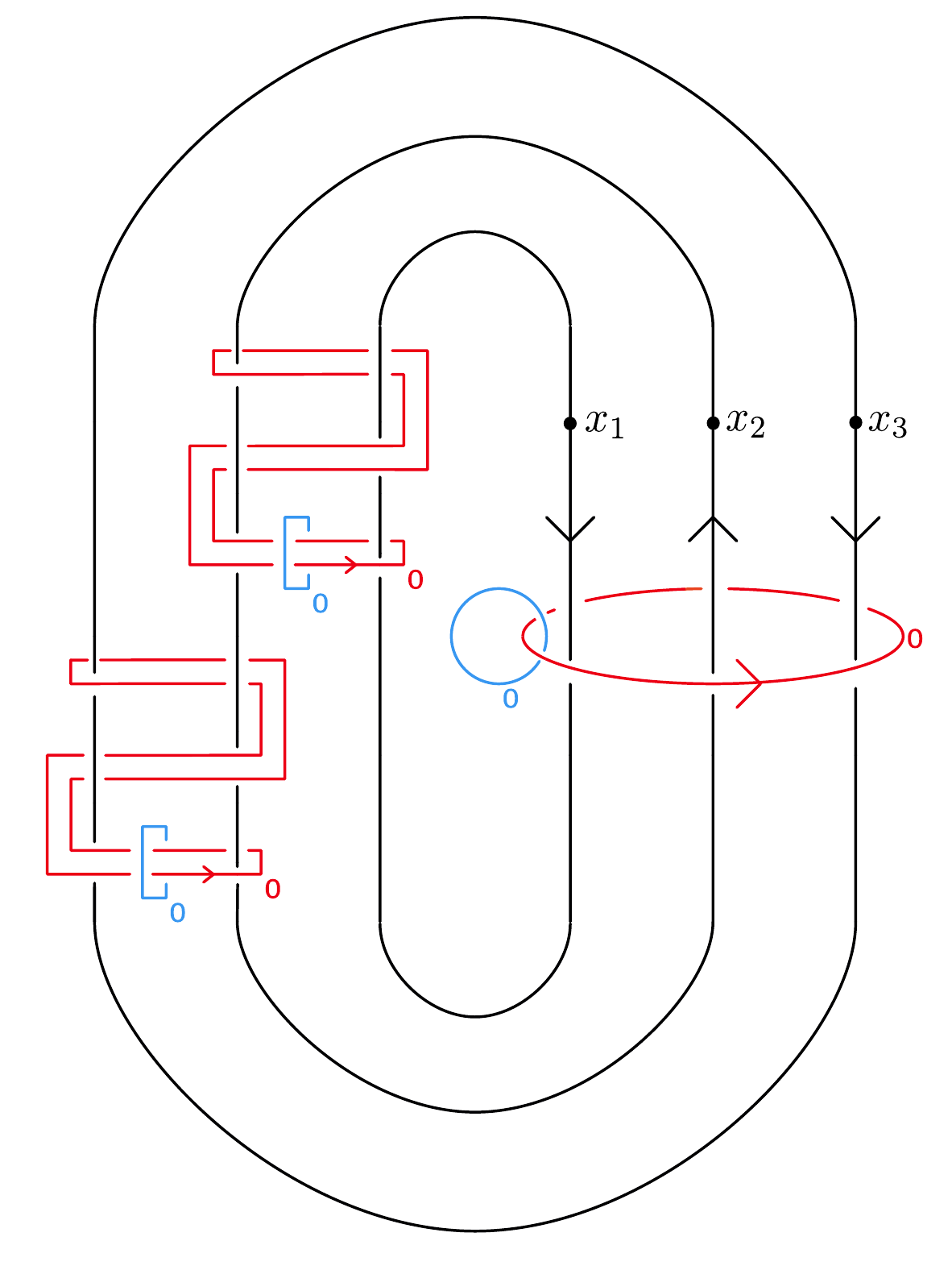}
    \caption{A handle decomposition of $S^n\times S^2\setminus int(\nu (K^n))$ consists of a single $(n+2)$-dimensional $0$-handle, three $1$-handles (black dotted $(n-1)$-spheres), three $2$-handles (red $1$-spheres), three $n$-handles (blue $(n-1)$-spheres) and two $(n+1)$-handles (not drawn).}
    \label{A8}
\end{figure}

\begin{remark}[A handle decomposition of $S^n\times S^2\setminus int(\nu (K^n))$] \label{rem:HDofCmplmnt} Each $i$-handle of $K^n$ determines an $(i+1)$-handle for $S^n\times S^2\setminus int(\nu (K^n))$ when the codimension is $2$. (See chapter 6.2 in \cite{gompf19994}.) We note that the number of $(n+1)$-handles for the complement induced by the $n$-handles for $K^n$ is one less than the number of $n$-handles and a dotted $(n-1)$-sphere means carving a properly embedded trivial $n$-ball in $B^{n+2}$ whose boundary is a dotted $(n-1)$-sphere, which is equivalent to attaching a $1$-handle. Therefore, the handle decomposition of $K^n$ in Remark \ref{rem:HDofKn} gives the handle decomposition of $S^n\times S^2\setminus int(\nu (K^n))$ in Figure \ref{A8}.
\end{remark}

\begin{proposition}\label{pro:FndmntlGpofCmplmnt}
$\pi_1(S^n\times S^2\setminus int(\nu (K^n)))$ is non-trivial.
    \begin{proof}
    Consider the handle decomposition of the complement $S^n\times S^2\setminus int(\nu (K^n))$ of $K^n$ in $S^n\times S^2$ in Figure \ref{A8}. A black dotted $(n-1)$-sphere and a red $1$-sphere represent a $1$-handle and a $2$-handle, respectively, so the fundamental group $\pi_1(S^n\times S^2\setminus int(\nu (K^n)))$ has the following presentation: 
        \begin{equation*}
        \langle x_1,x_2,x_3|x_1x_2x_1x_2^{-1}x_1^{-1}x_2^{-1}=1,x_2^{-1}x_3^{-1}x_2^{-1}x_3x_2x_3=1,x_1^{-1}x_2x_3^{-1}=1 \rangle.
        \end{equation*} 
        Delete $x_3$ by using the third relation $x_3=x_1^{-1}x_2\Leftrightarrow x^{-1}x_2x_3^{-1}=1:$
        \begin{equation*}
        \langle x_1,x_2|x_1x_2x_1x_2^{-1}x_1^{-1}x_2^{-1}=1,x_2^{-2}x_1x_2^{-1}x_1^{-1}x_2^2x_1^{-1}x_2=1\rangle.
        \end{equation*}
        Simplify the second relation by multiplying both sides by $x_2$ on the left and $x_2^{-1}$ on the right:
        \begin{equation*}
        \langle x_1,x_2|x_1x_2x_1x_2^{-1}x_1^{-1}x_2^{-1}=1,x_2^{-1}x_1x_2^{-1}x_1^{-1}x_2^2x_1^{-1}=1\rangle.
        \end{equation*}
        Use the substitution $x_1=ab^{-1}$ and $x_2=b^2a^{-1}$, in other words let $a=x_1x_2x_1$ and $b=x_2x_1$:
        \begin{equation*}
        \langle a,b|a^2b^{-3}=1,ab^{-2}ab^{-1}ab^{-1}a^{-1}b^2a^{-1}b^2a^{-1}ba^{-1}=1\rangle\\.
        \end{equation*}
        Simplify the second relation by multiplying both sides by $a^{-1}$ on the left and $a$ on the right :
        \begin{equation*}
        \langle a,b|a^2b^{-3}=1,b^{-2}ab^{-1}ab^{-1}a^{-1}b^2a^{-1}b^2a^{-1}b=1\rangle\\.
        \end{equation*}
        Simplify the second relation by multiplying both sides by $b$ on the left and $b^{-1}$ on the right:
        \begin{equation*}
        \langle a,b|a^2b^{-3}=1,b^{-1}ab^{-1}ab^{-1}a^{-1}b^2a^{-1}b^2a^{-1}=1\rangle\\.
        \end{equation*}
        Include the relation $a^2=b^3=1:$
        \begin{equation*}
        \langle a,b|a^2=b^{3}=1,b^{-1}ab^{-1}ab^{-1}a^{-1}b^2a^{-1}b^2a^{-1}=1\rangle\\.
        \end{equation*}
        Multiply both sides by $a$ on the right in the second relation:
        \begin{equation*}
        \langle a,b|a^2=b^{3}=1,b^{-1}ab^{-1}ab^{-1}a^{-1}b^2a^{-1}b^2=a\rangle\\.
        \end{equation*}
        Multiply both sides by $b$ on the right and use $b^3=1$ in the second relation:
        \begin{equation*}
        \langle a,b|a^2=b^{3}=1,b^{-1}ab^{-1}ab^{-1}a^{-1}b^2a^{-1}=ab\rangle\\.
        \end{equation*}
        Multiply both sides by $a$ on the right in the second relation:
        \begin{equation*}
        \langle a,b|a^2=b^{3}=1,b^{-1}ab^{-1}ab^{-1}a^{-1}b^2=aba\rangle\\.
        \end{equation*}
        Multiply both sides by $b$ on the right and use $b^3=1$ in the second relation:
        \begin{equation*}
        \langle a,b|a^2=b^{3}=1,b^{-1}ab^{-1}ab^{-1}a^{-1}=abab\rangle\\.
        \end{equation*}
        Multiply both sides by $a$ on the right in the second relation:
        \begin{equation*}
        \langle a,b|a^2=b^{3}=1,b^{-1}ab^{-1}ab^{-1}=ababa\rangle\\.
        \end{equation*}
        Multiply both sides by $b$ on the right in the second relation:
        \begin{equation*}
        \langle a,b|a^2=b^{3}=1,b^{-1}ab^{-1}a=ababab\rangle\\.
        \end{equation*}
        Multiply both sides by $a$ on the right and use $a^2=1$ in the second relation:
        \begin{equation*}
        \langle a,b|a^2=b^{3}=1,b^{-1}ab^{-1}=abababa\rangle\\.
        \end{equation*}
        Multiply both sides by $b$ on the right in the second relation:
        \begin{equation*}
        \langle a,b|a^2=b^{3}=1,b^{-1}a=abababab\rangle\\.
        \end{equation*}
        Multiply both sides by $a$ on the right and use $a^2=1$ in the second relation:
        \begin{equation*}
        \langle a,b|a^2=b^{3}=1,b^{-1}=ababababa\rangle\\.
        \end{equation*}
        Multiply both sides by $b$ on the right in the second relation:
        \begin{equation*}
        \langle a,b|a^2=b^{3}=1,1=ababababab\rangle\\.
        \end{equation*}
        Simplify the relations and get the following presentation:
        \begin{equation*}
        \langle a,b|a^2=b^{3}=(ab)^5=1\rangle\\,
        \end{equation*}which is isomorphic to the alternating group $A_5$ of degree $5$. Therefore, $\pi_1(S^n\times S^2\setminus int(\nu (K^n)))$ is non-trivial.
    \end{proof}
\end{proposition}

\begin{corollary}\label{cor:NIstpctoSnXpt}
$K^n$ is not isotopic to $S^n\times \{y_0\}$ in $S^n\times S^2.$
    \begin{proof}
    Suppose that $K^n$ is isotopic to $S^n\times \{y_0\}$. Then $S^n\times S^2 \setminus int(\nu (K^n))$ is diffeomorphic to $B^n\times S^2,$ so $\pi_1(S^n\times S^2\setminus int(\nu (K^n)))$ is trivial, which is a contradiction to Proposition \ref{pro:FndmntlGpofCmplmnt}.
    \end{proof}
\end{corollary}

We now construct a contractible $(n+3)$-manifold by using the $n$-knot $K^n.$

\begin{definition}\label{def:MainMfld}
Let $K^n$ be the $n$-knot in $S^n\times S^2$ in Definition \ref{dfn:theNknt}. Let $\phi:S^n\times B^2\hookrightarrow S^n\times S^2=\partial(S^n\times B^3)$ be an embedding such that $\phi(S^n\times\{0\})=K^n.$ We define $X_{K^n}:=S^n\times B^3\cup_{\phi}B^{n+1}\times B^2$ to be the $(n+3)$-manifold obtained from $S^n\times B^3$ by attaching a single $(n+1)$-handle along $\phi.$ 
\end{definition}

\begin{remark}
There is a unique framing of the attaching sphere $\phi(S^n\times\{0\})=K^n$ because $\pi_n(SO(2))$ is trivial when $n\geq2$. Therefore $X_{K^n}$ is uniquely determined by the isotopy class of $K^n.$
\end{remark}

\begin{proposition}\label{pro:XprdctIntvldiffeotoball}
$X_{K^n}\times B^1$ is diffeomorphic to $B^{n+4}.$
    \begin{proof}
    Let $\phi:S^n\times B^2\hookrightarrow S^n\times S^2$ be the embedding such that $\phi(S^n\times\{0\})=K^n$ in Definition \ref{dfn:theNknt} and $\Phi:S^n\times B^2\times B^1\hookrightarrow S^n\times S^2\times B^1$ be an embedding defined by $\Phi(x,y,t)=(\phi(x,y),t).$ Then $X_{K^n}\times B^1=(S^n\times B^3\cup_{\phi} B^{n+1}\times B^2)\times B^1\cong S^n\times B^3\times B^1\cup_{\Phi} B^{n+1}\times B^2\times B^1.$ By Proposition \ref{pro:istpctoSnXptXpt}, $\Phi(S^n\times\{0\}\times\{0\})=K^n\times\{0\}$ is isotopic to $S^n\times\{y_0\}\times\{0\}$ in $S^n\times S^2\times B^1 \subset\partial(S^n\times B^3\times B^1)$  so the attaching sphere of the $(n+1)$-handle intersects the belt sphere of the $n$-handle geometrically once and $S^n\times B^3\times B^1\cup_{\Phi} B^{n+1}\times B^2\times B^1\cong B^{n+4}.$ Therefore $X_{K^n}\times B^1\cong B^{n+4}.$
    \end{proof}
\end{proposition}

\begin{remark}[A handle decomposition of $\partial X_{K^n}$]\label{rem:HDofbndry} $\partial X_{K^n}$ is obtained from $S^n\times S^2$ by surgery along $K^n,$ i.e., $\partial X_{K^n}=(S^n\times S^2\setminus int(\nu(K^n))) \cup_{\phi|_{S^n\times S^1}} B^{n+1}\times S^1,$ where $\phi(S^n\times B^2)=\nu (K^n).$ We can consider $B^{n+1}\times S^1$ as the union of an $(n+2)$-dimensional $(n+1)$-handle and an $(n+2)$-handle. Therefore a handle decomposition of $X_{K^n}$ is obtained from the handle decomposition of $S^n\times S^2\setminus int(\nu(K^n))$ in Remark \ref{rem:HDofCmplmnt} by attaching an $(n+1)$-handle and an $(n+2)$-handle. Therefore $\partial X_{K^n}$ admits a handle decomposition with a $0$-handle, three $1$-handles, three $2$-handles, three $n$-handles, three $(n+1)$-handles and an $(n+2)$-handle.
\end{remark}

\begin{proposition}\label{pro:FndmntlGrSame}
$\pi_1(\partial X_{K^n})\cong\pi_1(S^n\times S^2\setminus int(\nu (K^n))).$
    \begin{proof}
    In Remark \ref{rem:HDofbndry},  $B^{n+1}\times S^1$ is the union of an $(n+1)$-handle and an $(n+2)$-handle, which does not affect the fundamental group of $\partial X_{K^n}.$ Therefore $\pi_1(\partial X_{K^n})\cong\pi_1(S^n\times S^2\setminus int(\nu (K^n))).$
    \end{proof}
\end{proposition}

\begin{corollary}\label{cor:nonsmplycnctdhlgysphre}
$\partial X_{K^n}$ is a non-simply connected homology $(n+2)$-sphere.
    \begin{proof}
    Clearly, $\partial X_{K^n}$ is a homology $(n+2)$-sphere because $X_{K^n}$ is a contractible manifold. Also, $\pi_1(\partial X_{K^n})\cong\pi_1(S^n\times S^2\setminus int(\nu (K^n)))$ is non-trivial by Proposition \ref{pro:FndmntlGpofCmplmnt} and \ref{pro:FndmntlGrSame}.
    \end{proof}
\end{corollary}

\begin{corollary}\label{cor:ctrcblebutnothomeo}
$X_{K^n}$ is contractible but not homeomorphic to $B^{n+3}.$
    \begin{proof}
        $X_{K^n}$ is contractible by Proposition \ref{pro:XprdctIntvldiffeotoball} but not homeomorphic to $B^{n+3}$ by Corollary \ref{cor:nonsmplycnctdhlgysphre}.
    \end{proof}
\end{corollary}

\begin{corollary}\label{cor:mniIntrsctnNmbrs}
There exists no $n$-knot $F$ in $S^n\times S^2$ such that $F$ is isotopic to $K^n$ and $|F\cap (\{x_0\}\times S^2)|<3.$
    \begin{proof}
    Suppose that there is an $n$-knot $F$ in $S^n\times S^2$ such that $F$ is isotopic to $K^n$ and $|F\cap (\{x_0\}\times S^2)|=1.$ Since $K^n$ and $F$ are isotopic, $X_{K^n}$ is diffeomorphic to $X_F$ which is obtained from $S^n\times B^3$ by attaching a single $(n+1)$-handle along $F.$ Since $|F\cap (\{x_0\}\times S^2)|=1$, $X_{K^n}$ is diffeomorphic to $B^{n+3}\cong X_F,$ which is a contradiction to Corollary \ref{cor:ctrcblebutnothomeo}. Suppose that there is an $n$-knot $F$ in $S^n\times S^2$ such that $F$ is isotopic to $K^n$ and $|F\cap (\{x_0\}\times S^2)|=2.$ Then, the algebraic intersection number $F\cdot (\{x_0\}\times S^2)$ is $0$ or $\pm2$, so $K^n\cdot (\{x_0\}\times S^2)$ is $0$ or $\pm2.$ This is a contradiction to $K^n\cdot (\{x_0\}\times S^2)=1.$ Therefore there exists no $n$-knot $F$ in $S^n\times S^2$ such that $F$ is isotopic to $K^n$ and $|F\cap (\{x_0\}\times S^2)|<3.$
    \end{proof}
\end{corollary}

We now prove our main theorems.

\mainone*
\begin{proof}
    Let $X_{K^n}$ be the $(n+3)$-manifold in Definition \ref{dfn:theNknt}. $X_{K^n}$ admits a handle decomposition with a $0$-handle, an $n$-handle and an $(n+1)$-handle by Definition \ref{dfn:theNknt}. $X_{K^n}$ is contractible but not homeomorphic to $B^{n+3}$ by Corollary \ref{cor:ctrcblebutnothomeo}.
\end{proof}

\begin{lemma}\label{lem:doubleissphere}
The double $DX_{K^n}=X_{K^n}\cup_{id}\overline{X_{K^n}}$ of $X_{K^n}$ is diffeomorphic to $S^{n+3},$ where $id:\partial X_{K^n}\rightarrow \partial X_{K^n}$ is an identity map.
    \begin{proof}
    $DX_{K^n}=X_{K^n}\cup_{id}\overline{X_{K^n}}\cong\partial (X_{K^n}\times B^1)\cong \partial(B^{n+4})=S^{n+3}$ by Proposition \ref{pro:XprdctIntvldiffeotoball}.
    \end{proof}
\end{lemma}

\maintwo*
    \begin{proof}
    By Lemma \ref{lem:doubleissphere}, $S^{n+3}\cong X_{K^n}\cup_{id} \overline{X_{K^n}}.$ Define an involution $\phi: S^{n+3}\rightarrow S^{n+3}$ switching copies of $X_{K^n}$ and fixing $\partial X_{K^n}.$
    \end{proof}

\bibliographystyle{plain} 
\bibliography{refs} 

\end{document}